\documentclass{amsart}

\usepackage{amsmath,amsthm,indentfirst}
\usepackage{amssymb}
\usepackage{amsfonts}
\usepackage{amscd}
\usepackage[latin1]{inputenc}
\theoremstyle{plain}
\newtheorem*{maintheorem*}{Main Theorem}
\newtheorem*{thm*}{Theorem}
\newtheorem*{thma*}{Theorem A}
\newtheorem*{thmaa*}{Theorem A'}
\newtheorem*{thmb*}{Theorem B}
\newtheorem*{thmo*}{Theorem 1.1}
\newtheorem*{thmc*}{Theorem C}
\newtheorem*{thmd*}{Theorem D}
\newtheorem*{thmf*}{Theorem 4.1}
\newtheorem*{remark*}{Remark}
\newtheorem*{conjecture*}{Conjecture}
\newtheorem*{prop*}{Proposition}
\newtheorem*{lem*}{Basic Lemma}
\newtheorem{thm}{Theorem}[section]
\newtheorem{cor}[thm]{Corollary}
\newtheorem{lem}[thm]{Lemma}
\newtheorem{prop}[thm]{Proposition}

\theoremstyle{definition}

\newtheorem*{proofc*}{Proof of Theorem C}

\def\bbz{\mathbb{Z}}

\def\bbr{\mathbb{R}}

\def\sl{\mathfrak{sl}}
\def\gfrak{\mathfrak{g}}

\def\hfrak{\mathfrak{h}}

\def\rfrak{\mathfrak{r}}

\def\scal{\mathcal{S}}
\def\acal{\mathcal{A}}
\def\vare{\varepsilon}

\def\vol{{\rm{vol}}}
\def\h{\hspace{1mm}}
\def\hh{\hspace{.5mm}}
\def\cpct{X_{{\rm{cpct}}}}
\def\cc{C_c^\infty}
\def\SL{{\rm{SL}}}

\title{\sc A special case of effective equidistribution with explicit constants}
\author{A.~Mohammadi}
\date{}

\address{Mathematics Dept., University of Chicago, Chicago, IL}
\email{amirmo@math.uchicago.edu}

\begin{document}
\maketitle
\begin{abstract}
An effective equidistribution with explicit constants for the isometry group of rational forms with signature $(2,1)$ is proved. As an application we get an effective discreteness  of Markov spectrum.  
\end{abstract}

\section{Introduction}\label{sec;intro}
Let $B(v)=2v_1v_3-v_2^2$ for any $v\in\bbr^3$ and let $H=\mbox{SO}(B).$ We denote $X=\SL_3(\bbr)/\SL_3(\bbz).$ Let $m$ be the $\SL_3(\bbr)$-invariant Haar measure on $X$ normalized so that $m(X)=1.$ Let $\acal\neq X$ be a closed subset of $X$ which is $H$-invariant. It was proved by S.~G.~Dani and G.~A.~Margulis~\cite{DM} that $\acal$ is a union of finitely many closed $H$-orbits. Indeed~\cite{DM} is a generalization of ideas and the results of Margulis' work in the proof of the Oppenheim conjecture ~\cite{Mar2, Mar4} and it is a special case of M.~Ratner's equidistribution Theorem~\cite{Rat3, Rat4, Rat5, Rat6} and~\cite{MT} for the proof of the measure classification and equidistribution theorems.

Recently M. Einsiedler, G.~Margulis and A.~Venkatesh, in a landmark paper~\cite{EMV}, proved a polynomially effective equidistribution theorem for closed orbits of semisimple groups. Their result is quite general and the proof is very involved. Our result in this paper gives explicit exponent in a very special case of the main theorem in~\cite{EMV}. 
Since we are dealing with a very special case the proof simplifies quite considerably. In this special case the calculation of the exponent in loc. cit. becomes much simpler and this is what is carried out here. Let us also mention that the main result of~\cite{EMV} and our result here are indeed effective version of special case of the results obtained by S.~Mozes and N.~Shah~\cite{MS} regarding the weak*-limits of measures invariant and ergodic with respect to a unipotent flow.  

As we mentioned, we borrow extensively from the results and techniques in~\cite{EMV}. There is however one main technical difference between this paper and~\cite{EMV}. In this paper we use  translates of a small piece of a unipotent orbit in $H$ by a particular torus in $H$ in order to get an ``effective ergodic" theorem. In~\cite{EMV} however the authors utilize such ergodic theorem for the action of unipotent subgroups. See also~\cite{E} where the action of unipotent groups is used to get measure classification for the action of semisimple groups. The main reason for this variant in the method is to get a ``bigger" exponent. The idea of using the torus in this way goes back to G.~A.~Margulis, where in an unpublished manuscript, he proved a topological version of the main result of this paper.   

Let $x_i\in X,$ $1\leq i\leq N$ be such that $Hx_i$ is closed for all $1\leq i\leq N.$ Let 

\begin{equation}\label{e;acal}\acal=\cup_{1\leq i\leq N}Hx_i\hspace{4mm}\mbox{and let}\hspace{4mm} \vol(\acal)=\sum_i\vol(Hx_i)\end{equation} 
We let $\mu_i$ denote the $H$-invariant probability measure on $Hx_i.$ We prove

\begin{thm}\label{thm;eff-equi}
Let the notation be as above. In particular, let $\acal$ be a union of finitely many closed $H$-orbits. There exist absolute constants $c, M_0>0$ such that if $\vol(\acal)=M>M_0$ then there exists $1\leq i_0\leq N$ such that for any $f\in\cc(X)$ we have
\begin{equation}\label{e;eff-equ}\left|\int_X f(x)d\mu_{i_0}(x)-\int_X f(x)dm(x)\right|< c\scal_{20}(f)\vol(\acal)^{-\delta}\end{equation} 
where $\scal_{20}$ is a certain Sobolev norm, see Section~\ref{sec;notation} for the definition. In particular $\delta=10^{-7}$ satisfies~\eqref{e;eff-equ}.
\end{thm}

An application of this theorem to effective discreteness of Markov Spectrum is given in the end of this paper, see Corollary~\ref{cor;markov}. 

\textbf{Acknowledgements.} We would like to thank Professor G.~A.~Margulis for his interest in this work and many enlightening conversations. We are in debt to M.~Einsiedler for reading the first draft of this paper and many helpful comments. We would like to thank E.~Lindenstrauss for helpful conversations on Markov Spectrum. We also thank the anonymous referee for many advantageous comments.  


\section{Notation and preliminaries}\label{sec;notation}
Throughout the paper we let $G=\mbox{SL}_3(\bbr)$ and $\Gamma=\SL_3(\bbz).$ We let
\begin{equation}\label{atut}a_t=\left(\begin{array}{ccc}e^t & 0 & 0\\ 0 & 1 & 0\\ 0 & 0 & e^{-t}\end{array}\right)\h\h\h u_s=\left(\begin{array}{ccc}1 & s & s^2/2\\ 0 & 1 & s\\ 0 & 0 & 1\end{array}\right)\end{equation}
The subgroups $\{a_t:\h t\in\bbr\}$ and $\{u_s:\h s\in \bbr\}$ are one parameter subgroups of $H=\mbox{SO}(B).$ Let $\gfrak=\sl_3(\bbr)$ (resp. $\hfrak$) denote the Lie algebra of $G$ (resp. $H$). We let $\|\h\|$ be a fixed Euclidean norm on $\gfrak$ and fix an orthonormal basis for $\gfrak$ with respect to this norm. We may write $\gfrak=\rfrak\oplus\hfrak$ where $\rfrak$ is the $\mbox{Ad}(H)$-invariant complement to $\hfrak.$ Further we let $\rfrak_0$ denote the centralizer of $\{u_s\}$ in $\rfrak$ and let $\rfrak_1$ denote the orthogonal complement of $\rfrak_0$ in $\rfrak$ with respect to $\|\h\|.$ For any $r\in\rfrak$ let $r_1$ (resp. $r_0$) denote the $\rfrak_1$ (resp. $\rfrak_0$) component of $r$ and define $Z_r=\frac{r_0}{\|r_0\|}$ with $Z_r=0$ if $r_0=0.$ Throughout $\exp$ will denote the exponential map from $\gfrak$ into $G.$ 


The Euclidean norm $\|\h\|$ on $\gfrak$ defines a right-invariant Riemannian metric on $G$ which in turn gives a metric on $X=G/\Gamma.$ The Riemannian metric on $G$ gives a volume form on subgroups of $G,$ this will be denoted by $\vol$ in the sequel. We will use Greek letter $\mu$ and $\mu_i$ to denote the $H$-invariant probability Haar measure on closed $H$-orbits.

Let $\Gamma=\mbox{SL}_3(\bbz)$ and $X=G/\Gamma.$ This is the space of unimodular lattices in $\bbr^3.$ For any $x\in X,$ we let $g_x\bbz^3$ denote the corresponding lattice in $\bbr^3$. And define $\alpha_1(x)=\max_{0\neq v\hh\in g_x\bbz^3}\frac{1}{|v|}.$ By Mahler's compactness criteria the set 
\begin{equation}\label{compactpart}\mathfrak{S}(R)=\{x\in X:\h \alpha_1(x)\leq R\}\end{equation} 
is compact. We will use the Greek letter $\mu$ and also $\mu_i$'s for $H$-invariant probability measures on closed $H$-orbits in $X$ and will save the letter $m$ for the $G$-invariant probability measure on $X.$ Recall that $H$ is a maximal subgroup of $G.$ In particular $H$ is not contained in any proper parabolic subgroup of $G.$ Using this and quantitative non-divergence results proved D.~Kleinbock and G.~Margulis~\cite{KM}, see also~\cite[Appendix B]{EMV}, we have the following 

\begin{lem}\label{lem;nondiv}
With the notation as above
\begin{itemize}
\item[(i)] There exists $R_1>1$ such that any $H$-orbit in $X$ intersects $\mathfrak{S}(R_1).$
\item[(ii)] There exists a constant $c$ depending on $R_1$ such that for any $H$-invariant measure $\mu$ on $X$ we have $\mu(\{x\notin \mathfrak{S}(R)\})\leq cR^{-1/2}$
\end{itemize}
\end{lem}
Let $R_0$ be such that $\mu(X\setminus\mathfrak{S}(R_0))<10^{-11}.$ We define $\cpct=\mathfrak{S}(R_0).$ This will be our favorite fixed compact set throughout this paper. 

We will define $\|\h\|$ on $G$ by letting $\|g\|$ to be the maximum of the absolute value of matrix coefficients of $g$ and $g^{-1}$. This should not cause any confusion with $\|\h\|$ on $\gfrak$ in the sequel.

The letter $f$ will denote an element of $C_c^\infty(X)$ throughout the sequel. If $f\in C_c^\infty(X)$ then $\|f\|_2$ (resp. $\|f\|_\infty$) denotes the $L^2$ norm (resp. $L^\infty$ norm) of $f.$ $G$ acts on these spaces by $g\cdot f(x)=f(g^{-1}x).$ For the sake of simplicity in notation if $\sigma$ is a measure on $X$ we will denote $\sigma(f)=\int_Xfd\hh\sigma.$ For any $g\in G$ we let $\sigma^g$ be the measure defined by $\sigma^g(f)=\sigma(g^{-1}f),$ for any $f\in C_c^\infty(X).$ 

We now recall the definition of a certain family of Sobolev norms which were also used in~\cite{EMV}. For any integer $d\geq0$ we let $\scal_d$ be the Sobolev norm on $X$ defind by
\begin{equation}\label{sobolev}\scal_d(f)^2=\sum_{\mathcal{D}}\|\alpha_1(\hh\cdot\hh)^d\mathcal{D}f\|_2^2
\end{equation}
where the sum is taken over all monomials of degree at most $d$ in the fixed basis for $\gfrak.$ 
We will need the following three properties of $\scal_d.$ Let $d\geq8$ then for any $g\in G$ and $f\in C_c^\infty(X)$ we have
\begin{itemize}
\item[(i)]$\scal_d(g\hh f)\leq c(d)\|g\|^{2d}\scal_d(f)$
\vspace{.7mm}
\item[(ii)]$\|f\|_\infty\leq c'(d)\scal_d(f)$
\vspace{.7mm}
\item[(iii)]$\|g\hh f-f\|_\infty\leq d(e,g)\scal_d(f)$
\end{itemize}
For a discussion of the Sobolev norm and the proofs of the above properties we refer to~\cite[section 5]{EMV}.

Let $\scal_d$ be as above and let $\vare>0$ we say the measure $\sigma$ is $\vare$-almost invariant under $g\in G$ if $|\sigma(g\hh f)-\sigma(f)|<\vare\scal_d(f).$ We say $\sigma$ is $\vare$-almost invariant under $G$ if it is $\vare$-almost invariant under all $g\in G$ with $\|g\|\leq 2.$

In the proof of Proposition~\ref{prop;genericset} we will need to use the notion of relative trace, see~\cite[Section 5.2]{EMV} and~\cite[Appendix 2]{BR} for a more extensive study of these ideas. Let $V$ be a complex vector space and $A$ and $B$ be two non-negative Hermitian forms on $V$ further assume that $B$ is positive definite. If $V$ is finite dimensional define the relative trace by ${\rm{Tr}}(A|B)={\rm{Tr}}(B^{-1}A).$ It follows from the definition that we have the following formula for the relative trace
\begin{equation}\label{rel-tr}{\rm{Tr}}(A|B)=\sum_i\frac{A(e_i)}{B(e_i)}\end{equation}
where $\{e_i\}$ is any orthonormal basis for $V$ with respect to $B.$ 
If $V$ is infinite dimensional define 
\begin{equation}\label{rel-tr2}{\rm{Tr}}(A|B)=\sup_{W\subset V}{\rm{Tr}}(A_W|B_W)\end{equation}
where the supremum is taken over all finite dimensional subspaces and $A_W,$ $B_W$ denote restriction of $A$ and $B$ to $W.$ Note that $\scal_d$ defines a Hermitian norm in $C_c^\infty(X).$ We have; for every $d$ there exists $d'>d$ such that
\begin{equation}\label{rel-tr3}{\rm{Tr}}(\scal_d^2|\scal_{d'}^2)<\infty\end{equation}
That is to say there exists an orthonormal basis $\{e_i\}$ for the completion of $C_c^\infty(X)$ with respect to $\scal_{d'}$ such that $\sum_i\scal_d(e_i)^2<\infty.$ 

\section{Producing effective extra invariant}\label{sec;extra-inv}
In this section we will use an effective ``ergodic" theorem, see proposition~\ref
{prop;genericset} and show that we can effectively produce elements outside $H$ which will move one closed orbit $Hx_i$ in $\acal$ ``close" to another closed orbit $Hx_j$ in $\acal,$ see proposition~\ref{prop;extrainv} below.

\textbf{Generic points.} Let $t\geq1.$ For any $x\in X$ define
\begin{equation}\label{dtf}D_t^f(x)=e^{7t/24}\int_0^{e^{-7t/24}}f(a_tu_{s}x)ds-\int_Xfd\mu\end{equation}
Let $t\in\bbr$ be a positive number a point $x\in X$ is called $t$-{\it generic} with respect to the Sobolev norm $\scal,$ if for any integer $n\geq t$ we have   
\begin{equation}\label{generic}|D_n^f(x)|\leq e^{-n/8}\scal(f)\end{equation}
As is clear from the definition above, generic points provide us with an effective ergodic theorem. The following proposition shows that ``most" points are generic.

\begin{prop}\label{prop;genericset}
There exists $d_0$ depending on $G$ and $H$ only such that for all $d\geq d_0$ the $\mu$-measure of points which are not $t$-generic with respect to $\scal_d$ is $O(e^{-t/24}).$ Furthermore $d=20$ satisfies this property.
\end{prop}

\begin{proof}
Let $f$ be a fixed function in $C_c^\infty(X).$ Let $\langle\h,\h\rangle$ denote the inner product on $L^2(X,\mu).$ Note that $\mbox{dim}\hh H=3.$ As a result of explicit estimate for the decay of matrix coefficients~\cite{KS} we have
\begin{equation}\label{expdecay}\left|\langle u_sf,f\rangle-\left(\int_Xfd\mu\right)^2\right|\leq (1+|s|)^{-0.7}\scal_{3}(f)^2\end{equation}
Let $A=\{(s_1,s_2)\in I\times I:\h |s_1-s_2|<e^{-7t/12}\}$ and $B=I\times I\setminus A,$ where $I=[0,e^{-\frac{7t}{24}}].$ Squaring $D_t^f$ we have 
\begin{align*}\int_X|D_t^f(x)|^2d\mu &=e^{\frac{7t}{12}}\int_I\int_I\langle a_tu_{s_1-s_2}a_{-t}f, f\rangle ds_1ds_2-\left(\int_Xfd\mu\right)^2 
\\ &=e^{\frac{7t}{12}}\left(\int_A\langle a_tu_{s_1-s_2}a_{-t}f, f\rangle ds_1ds_2+\int_B\langle a_tu_{s_1-s_2}a_{-t}f, f\rangle ds_1ds_2\right)\\ &-\left(\int_Xfd\mu\right)^2\leq 3e^{-7t/24}\h\scal_{3}(f)^2\end{align*}
Hence we have 
\begin{equation}\label{chebechev}\mu(\{x\in X:\h |D_t^f(x)|>r\})\leq 3{r^{-2}e^{-7t/24}\h\scal_{3}(f)^2}\end{equation}
The rest of the proof is mutandis mutadus of the proof of proposition 9.2 in~\cite{EMV}. Let us recall this proof. First choose $d>d'>3$ such that $\mbox{Tr}(\scal_{d'}^2|\scal_{d}^2)$ and $\mbox{Tr}(\scal_{3}^2|\scal_{d'}^2)$ are both finite, see Section~\ref{sec;notation} and references there. Now choose an orthonormal basis $\{e_i:\h i\geq 1\}$ for the completion of $C_c^\infty(X)$ with respect to $\scal_d$ which is also orthogonal with respect to $\scal_{d'},$ note that this choice can be made thanks to the spectral theorem.

Let $c=\left(\sum_i\scal_{d'}(e_i)^2\right)^{-1/2}$  and define 
\begin{equation}\label{genericset}E=\bigcup_{n\geq t, i\geq 1}\{x\in X:\h |D_n^{e_i}(x)|\geq 3\hh c\hh e^{-n/8}\scal_{d'}(e_i)\}\end{equation} 
Using~(\ref{chebechev}) and the fact that $\mbox{Tr}(\scal_{3}^2|\scal_{d'}^2)$ is finite, we have 
\begin{equation}\label{genericmeasure}\mu(E)\leq \frac{\mbox{Tr}(\scal_{3}^2|\scal_{d'}^2)}{c^2}e^{-n/24}\end{equation} 
Now using linearity of $D_n$ and Cauchy-Schwarz inequality we have; If $f=\sum_if_ie_i\in C_c^\infty(X)$ then for any $n\geq t$ and any $x\notin E$ we have $|D_n^f(x)|\leq e^{-n/8}\scal_d(f).$ This combined with~(\ref{genericmeasure}) finishes the proof of the proposition. To see the last conclusion note that $\mbox{dim}(G)=8.$ Hence it follows from properties of the Sobolev norm~\cite[section 5]{EMV} that $d'=3+8$ and $d=11+8$ satisfy the properties we need in the above argument.   
\end{proof}


\textbf{Producing extra invariants.}
Let $\mu_1$ and $\mu_2$ be two $H$-invariant probability measures on $X.$ We  use Proposition~\ref{prop;genericset} to show that if the supports of these measures come ``very close'' to each other then we may effectively produce elements in the direction of $\rfrak_0$ which move $\mu_1$ ``close'' to $\mu_2$. This idea is by no means new and has been used by several people, to give two important examples we refer to G.~Margulis's proof of the Oppenheim conjecture and M.~Ratner's proof of the measure rigidity conjecture. Here using generic points we will give effective version of this phenomena. This effective form is one of the main ingredients in~\cite{EMV} as well. The following is the precise statement

\begin{prop}\label{prop;extrainv}
Let $\scal_d$ be as in the Proposition~\ref{prop;genericset} and let $\iota>0$ be a constant. Suppose that $\mu_1$ and $\mu_2$ are two $H$-invariant measures and that $y_1,y_2\in X$ are such that $y_2=\exp(r)y_1,$ where $r\in\rfrak$ for some $\|r\|<1$ is implicitly assumed to be small enough and that $\|r_0\|\geq\iota\|r\|.$ Moreover assume $y_i$'s are $\log(1/\sqrt{\|r\|})$-generic for $\mu_i$ with respect to $\scal_d$ for $i=1,2$. Then for any $\tau_1\geq1$ there exists a constant $c=c(d,\tau_1,\iota)$ such that
\begin{equation}\label{almostinv1}|\mu_1^{\exp(\tau Z_r)}(f)-\mu_2(f)|\leq c\hh \|r\|^{-1/336}\scal_d(f)\hspace{2.5mm}\mbox{for all}\h\h 0\leq\tau\leq\tau_1\end{equation}
where $Z_r$ is defined as in Section~\ref{sec;notation}.
\end{prop}

The proof of this proposition is essentially based on the following fact; If $r\in\gfrak$ is ``generic enough" then $\mbox{Ad}(a_t)r$ will be ``close" to $\rfrak_0$. This indeed is a consequence of the fact that $\rfrak_0$ is the direction of maximum expansion for the conjugation action of $a_t$ for $t>0.$ The genericity condition we impose is $\|r_0\|\geq\iota\|r\|.$ Note that the adjoint action of a ``small piece" of $u_s$ leaves $r$ ``almost" fixed. We now flow by $a_tu_s$ from $y_2=\exp(r)y_1.$ The above discussion implies that this orbit is ``close" to a translate of the orbit of $y_1$ under this flow, by an element in $\exp(\rfrak_0).$ The proposition then follows from the definition of a generic point and some continuity arguments.   

\begin{proof}
Let us start with the following simple calculation; If $r=r_0+r_1\in\rfrak$ then 
\[\mbox{Ad}(u_s)(r)=r+s\hh p(r,s),\hspace{2mm}\mbox{where}\h\h p(r,s)\in\rfrak\h\h\mbox{and}\h\h\|p(r,s)\|\leq c(s)\|r\|\] 
and $c(s)$ is a constant depending on $s$ and the norm $\|\h\|$. In particular if $|s|\leq1,$ then $c(s)$ is an absolute constant and we will omit it from the notation from now. Hence for any positive integer $n>0$ and any $0\leq s\leq e^{-7n/24}$ we have 
\[\|(\mbox{Ad}(u_s)(r))-r\|\leq e^{-7n/24}\|r\|\]  
Let now $r$ be as in the statement of the proposition. For any positive integer $n>0$
we have $\|\mbox{Ad}(a_n)r_1\|\leq e^n\|r\|.$ Let $n>0$ be an integer so that $1\leq e^{2n}\|r_0\|\leq e^4.$ Our assumption, $\|r_0\|\geq\iota\|r\|,$ implies that $e^n\|r\|\leq c({\iota})e^{-n}.$ We thus have
\begin{equation}\label{r0direction}\|\mbox{Ad}(a_nu_s)r-e^{2n}r_0\|\leq e^{n}\|r\|\leq c(\iota)e^{-n}\end{equation}
Using~\eqref{r0direction} and the property (iii) of the Sobolev norm $\scal_d,$ we have 
\begin{equation}\label{polycons}\|f(a_nu_s\exp(r)y_1)-f(\exp(e^{2n}r_0)a_nu_sy_1)\|_\infty\leq c(\iota)e^{-n}\scal_d(f)
\end{equation}
Recall now that $y_i$'s are $\log(1/\sqrt{\|r\|})$-generic for $\mu_i.$ This together with  our choice of the integer $n$ gives
\begin{equation}\label{e;extrainv}|e^{7n/24}\int_0^{e^{-7n/24}}f(a_nu_sy_i)ds-\int_Xfd\mu_i|\leq e^{-n/8}\scal_d(f)\h\h\mbox{for}\h\h i=1,2
\end{equation}
We now combine~(\ref{polycons}) and~(\ref{e;extrainv}) and get
\begin{equation}\label{e;extrainv1}\left| e^{7n/24}\int_0^{e^{-7n/24}}f(a_nu_s\exp(r)y_1)ds-\mu_1^{\exp(e^{2n}r_0)}(f)\right|\leq 2c(\iota)\hh e^{-n/8}\scal_d(f)\end{equation}
The fact that $y_2=\exp(r)y_1,$~\eqref{e;extrainv} and~\eqref{e;extrainv1} now imply 
\begin{equation}\label{almostinv2}|\mu_1^{\exp(e^{2n}r_0)}(f)-\mu_2(f)|\leq  c\hh e^{-n/8}\scal_d(f)
\end{equation}
for some constant $c$ depending on $\iota.$
Which proves~(\ref{almostinv1}) for some $\tau_0\geq 1.$ To prove the proposition for all values of $0\leq\tau\leq\tau_1$ note that
\begin{equation}\label{e;iteration}|\mu_1^{a_t\exp(\tau_0 Z_r)a_{-t}}(f)-\mu_2(f)|=|\mu_1^{a_t\exp(\tau_0 Z_r)}(f)-\mu_2^{a_t}(f)|\leq c\hh e^{-n/8}\scal_d(a_t\cdot f) \end{equation}
Recall that we may let $d=20$ in above calculations, hence property (i) of $\scal_d$ gives $\scal_d(a_t\cdot f)\leq e^{40t}\scal_d(f).$ This and~(\ref{e;iteration}) give the result for $e^{-n/168}\leq\tau\leq\tau_1.$ The case $0\leq\tau<e^{-n/168}$ follows using continuity argument, see property (iii) of the Sobolev norm. 
\end{proof}


\textbf{Finding generic points close to each other.}
Let the notation be as before. In particular $\mathcal{A}$ is a subset of $X$ which consists of a finite number of closed $H$-orbits and let $M=\vol(\acal)$ be as in~\eqref{e;acal}. In order to be able to apply Proposition~\ref{prop;extrainv} we need to produce close by generic points. The following proposition provides us with such points. This proposition is essentially a result of the Dirichlet's pigeon hole principle.

\begin{prop}\label{prop;genptsclose}
There exist absolute constants $c_{},\iota>0$ and a Sobolev norm $\scal_d$ with the following property; there are $Hx_i\in\acal$ and points $y_i\in Hx_i$ for $i=1,2,$ so that $y_2=\exp(r)y_1,$ where $r\in\rfrak$ with $\|r\|<c_{}M^{-1/5}$ and $\|r_0\|\geq\iota\|r\|.$ Moreover $y_i$ is $\log(1/\sqrt{\|r\|})$-generic for $\mu_i$ with respect to $\scal_d,$ where $\mu_i$ is the $H$-invariant probability measure on the closed orbit $Hx_i$ for $i=1,2.$ 
\end{prop}

\begin{proof}
This is virtually a special case of proposition 14.2 in~\cite{EMV}. We reproduce the proof in here. We first work with a single closed $H$-orbit. So let $Hx$ be a closed $H$ orbit and $\mu$ be the $H$-invariant probability measure on $Hx.$ Let $t$ be a large number so that $e^{-t/24}<10^{-11}$ and let $E'_{Hx}$ be the set of $t$-generic points obtained by Proposition~\ref{prop;genericset} and $E_{Hx}=E'_{Hx}\cap\cpct$. For any $\delta>0$ let $\rfrak_\delta$ (resp. $\hfrak_\delta$) denote the ball of radius $\delta$ in $\rfrak$ (resp. $\hfrak$) around the origin with respect to the norm $\|\h\|.$ Let $\delta_0$ be small enough such that for any $\delta<\delta_0$ and any $z\in\cpct$ the natural map from $\rfrak_\delta\times\hfrak_\delta$ to $X$ given by $\pi_z(r,h)=\exp(r)\exp(h)z$ is a diffeomorphism, indeed we are also assuming that the restriction of the exponential map to $\rfrak_\delta\times\hfrak_\delta$ into $G$ is a diffeomorphism. We let $\Omega=\exp(\hfrak_{\delta_0/2}).$ Define the following function on $X$
\begin{equation}\label{desityfunc}\phi(z)=\frac{1}{\vol(\Omega)}\int_\Omega\chi_E(hz)\hh d\hh\vol(h)\end{equation}
We have $\int_X\phi(z)d\mu=\mu(E_{Hx})>1-\frac{2}{10^{11}}.$ Hence the set $F_{Hx}=\{z\in E_{Hx}:\h \phi(z)>0.99\}$ has measure at least $0.9.$ Define
\begin{equation}\label{goodset}\mathcal{F}=\mathcal{F}(\acal)=\bigcup_{Hx\in\mathcal{A}}F_{Hx}\end{equation}
Note that $\vol(\mathcal{F})\geq 0.9M.$ Let now $\delta<\delta_0$ be given, this will be determined in terms of $M$ later. We let $B_\delta z=\pi_z(\rfrak_\delta\times\hfrak_\delta)$ i.e. the image of $\pi_z$. We may cover $\mathcal{F}$ by finitely many sets of the form $B_\delta z$ with finite multiplicity $c''$ independent of $\delta,$ in this covering the center $z$ of each $B_\delta z$ is assumed to be in $\mathcal{F}.$ Now using Dirichlet's pigeon hole principle we have; There is an absolute constant $c'$ depending on $c''$ such that if we let $\delta=c'M^{-1/5}$ and if $M$ is large enough such that $\delta<\delta_0$ then the following holds; There are $Hx_1$ and $Hx_2$ in $\mathcal{A}$ and $y'_i\in F_{Hx_i}$ for $i=1,2$ such that $y'_i\in B_\delta z$ for some $z$ and $y_1\neq hy_2$ for any $h\in\Omega.$ We now want to perturb $y'_i$'s within $F_{Hx_i}$ to guarantee that they satisfy the conclusion of lemma. First let us remark on the following simple statement; There is a constant $\iota>0$ such that
\begin{equation}\label{conjgeneric}\vol\{h\in\Omega:\h \frac{\|(\mbox{Ad}(h)r)_0\|}{\|r\|}\leq\iota\}<\vol(\Omega)/2
\end{equation}
This indeed is a result compactness argument and the fact that the $\mbox{Ad}$-representation of $H$ on $\rfrak$ is irreducible. Now if we apply the implicit function theorem and use the fact that $\phi(y_i')>0.99,$ we can find $h_i\in\Omega$ such that $h_iy_i'\in F_{Hx_i}$ and $h_2y_2=\exp(r)h_1y_1'$ where $\|r\|\leq cM^{-1/5},$ for some absolute constant depending on $c',$ moreover we have $\|r_0\|>\iota\|r\|.$ So $y_i=h_iy_i'$ for $i=1,2$ satisfy the conclusion of the proposition. 
\end{proof}


\section{Proof of Theorem~\ref{thm;eff-equi}}\label{sec;proof}

We will use the results from the last section here and prove the main theorem. We first need the following statement.

\textbf{$G$-almost invariant measures are close to the Haar measure.} This follows from the fact that the action of $G$ on $X$ has spectral gap. This is indeed proposition 15.1 in~\cite{EMV}. Before stating the proposition let us recall here that the conclusion of Proposition~\ref{prop;extrainv} holds for $d=20.$ This will be used when we iterate the $\vare$-almost invariance inequality. We have the following

\begin{prop}\label{prop;closetohaar}
Let $\sigma$ be a probability measure on $X$ which is $\vare$-almost invariant under $G$ with respect to $\scal_d.$ Then there is some constant $C$ such that 
\begin{equation}\label{e;closetohaar}|\sigma(f)-m(f)|\leq C \vare^{0.0002}\scal_d(f)
\end{equation}
where $f$ is in $C_c^\infty(X)$ as before. 
\end{prop}

\begin{proof}
We will reproduce the proof given in~\cite{EMV}. We begin by fixing some notation. For any real number $r\geq4$ we let $A_r=\{a=\mbox{diag}(a_1,a_2,a_3): 1\leq\frac{a_1}{a_2}\leq r,\h 1\leq\frac{a_2}{a_3}\leq r\}.$ It is an exercise in linear algebra to see that for any $a, b\in A_r$ with $d(a,e)>r/2$ and $d(b,e)>r/2$ the measure of $\{k\in K:\h d(akb,e)<r/2\}$ is at most 1/2 the measure of $K.$ Let $\chi$ be the characteristic function of $KA_6K$ normalized to have mean one. It follows from this fact and~\cite{Oh}, see in particular the main theorem and the discussion in section 8 in loc. cit., that for any $f\in C_c^\infty(X)$ with $m(f)=0$ we have $\|\chi\star f\|_2\leq(\frac{1}{2}+\frac{\delta}{2}){\|f\|_2}$ where $\delta\leq\frac{11.09+\log 36}{6}<1/2.$ Hence we get
\begin{equation}\label{spectralgap}\|\chi\star f\|_2\leq\frac{3}{4}\|f\|_2
\end{equation}
It suffices to show~(\ref{e;closetohaar}) for function $f$ with $m(f)=0.$ We have $|\sigma(\chi\star f)-\sigma(f)|<c(d)\hh\vare\hh\scal_d(f)$ as $\sigma$ is $\vare$-almost invariant under $G.$ Iterating this, see section 8.2 in~\cite{EMV}, we have
\begin{equation}\label{iteration}|\mu(\chi^n\star f)=\mu(f)|<c(d)6^{40n}\hh\vare\hh\scal_d(f)
\end{equation}
Note that we have the trivial bound $|\chi^n\star f(x)|\leq \|f\|_\infty\leq\scal_d(f).$ Note also that the fibers of the map $x\mapsto gx$ is at most of size $c\alpha_1(x)^8$ for a constant depending on $r.$ Now using Cauchy-Schwarz inequality we have
\begin{equation}\label{convbound}|\chi\star f(x)|\leq \alpha_1(x)^8\|f\|_2
\end{equation}
Hence we get the more interesting bound $|\chi^n\star f(x)|\leq{\alpha_1(x)^8}\frac{3^n}{4^n}\|f\|_2.$ If we combine these inequalities now, we obtain
\begin{equation}\label{upperbound}\sigma(f)\leq C(d)\scal_d(f)\left(6^{40n}\vare+\int_X\min(1,{\alpha_1(x)^8}\frac{3^n}{4^n}\right)
\end{equation}
Divide $X$ into $\mathfrak{S}(R)$ and its complement. Using~(\ref{upperbound}) and Lemma~\ref{lem;nondiv} we have 
\begin{equation}\label{lastbound}\sigma(f)\leq C(d)\scal_d(f)(6^{40n}\vare+{R^8}\frac{3^n}{4^n}+R^{-1/2})
\end{equation}
Choosing $R$ and $n$ such that the terms in the parenthesis are of the same size and we get $|\sigma(f)|<\vare^{\kappa_2}\scal_d(f)$ and $\kappa_2=\frac{1/17\log (3/4)}{1/17\log(3/4)-40\log 6}.$ Hence $\kappa_2\geq 0.0002.$ 
\end{proof}


\textbf{Proof of Theorem~\ref{thm;eff-equi}.} 
Let $\acal$ be as in the statement of Theorem~\ref{thm;eff-equi} and let $M=\vol(\acal).$ We apply Proposition~\ref{prop;genptsclose}. Hence we find; $Hx_i\in\acal$ and points $y_i\in Hx_i$ for $i=1,2,$ so that $y_2=\exp(r)y_1,$ where $r\in\rfrak$ with $\|r\|<c_{}M^{-1/5}$ and $\|r_0\|\geq\iota\|r\|.$ Moreover $y_i$ is $\log(1/\sqrt{\|r\|})$-generic for $\mu_i$ with respect to $\scal_d,$ where $\mu_i$ is the $H$-invariant probability measure on the closed orbit $Hx_i$ for $i=1,2.$ We now apply Proposition~\ref{prop;extrainv} to these $\mu_i$'s and $y_i$'s and get
\begin{equation}\label{almostinv3}|\mu_1^{\exp(\tau Z_r)}(f)-\mu_2(f)|\leq c(d,\tau_0,\iota)\hh M^{-1/1680}\scal_d(f)\hspace{2.5mm}\mbox{for all}\h\h 0\leq\tau\leq\tau_0
\end{equation}

From this we have the following

\textit{Claim.}
There is an absolute constant $c$ such that for any $g\in G$ with $d(1,g)\leq2$ we have
\begin{equation}\label{eq;Ginv}|\mu_1^g(f)-\mu_1(f)|\leq M^{-1/1680}\scal_d(f)
\end{equation}
In other words $\mu_1$ is $M^{-1/1680}$-almost invariant under $G$ with respect to $\scal_d.$

\textit{Proof of the claim.}
This indeed is a corollary of~(\ref{almostinv3}) above and the fact that $H$ is a maximal subgroup of $G$. Let $H_1=\exp(-Z_r)H\exp(Z_r)$ and let $g\in H_1.$ Now using ~(\ref{almostinv3}), $H$-invariance of $\mu_2$ and properties of the Sobolev norm we have 
\begin{equation}\label{almostinv4}|\mu_1^{g}(f)-\mu_1(f)|\leq c(g, d)\hh M^{-1/1680}\scal_d(f)
\end{equation}
and $c(g, d)$ is uniform on compact sets. Let $\mathcal{Y}_\iota=Y_1\times\cdots\times Y_\iota,$ where each $Y_i$ is either $H$ or $H_1.$ As $G$ is generated by $H$ and $H_1$ it follows from standard facts about algebraic groups that there is $\iota$ such that the product map from $\mathcal{Y}_\iota$ to $G$ is dominant. This plus~(\ref{almostinv4}) and $H$-invariance of $\mu_1$ finish the proof.

Thanks to this claim we may now apply Proposition~\ref{prop;closetohaar} to the measure $\mu_1.$ Recall that $d$ maybe taken to be $20.$ We have
There is an absolute constant $c$ such that
\begin{equation}\label{e;muclosehaar}|\mu_1(f)-m(f)|\leq c\hh M^{-\delta}\scal_d(f)
\end{equation}
where we may take $\delta=\frac{1}{84}\times10^{-5}.$

\section{Effective discreteness of Markov spectrum}\label{sec;markov}
Our main theorem has an application in giving effective estimates for discreteness of Markov spectrum. let us first recall the definition and some known facts, see~\cite{Mar6} for a more detailed discussion.   

Let $\mathcal{Q}_n$ denote the set of all nondegenerate indefinite quadratic forms in $n$ variables. For any $Q\in\mathcal{Q}_n$ let $m(Q)=\inf\{|Q(v)|:\h v\in\bbz^n\setminus\{0\}\}$ and let $d(Q)$ denote the discriminant of $Q.$ Let $\mu(Q)=m(Q)^n/d(Q).$ It is clear that $\mu(Q)=\mu(Q')$ if $Q$ and $Q'$ are equivalent. Let $M_n=\mu(\mathcal{Q}_n).$ This is called the $n$-dimensional Markov spectrum. It follows from Mahler's compactness criteria that $M_n$ is bounded and closed. 

In 1880, A.~Markov~\cite{Mark} described $M_2\cap (4/9,\infty)$ and the corresponding quadratic forms. It follows from this description that $M_2\cap(4/9,\infty)$ is a discrete subset of $(4/9,4/5]$ and for any $a>4/9,$ $(a,\infty)\cap M_2$ is a finite set. On the other hand, $M_2\cap(0,4/9]$ is uncountable and has a quite complicated topological structure.

It follows from Meyer's theorem and Margulis' proof of the Oppenheim conjecture \cite{Mar2} that $M_n=\{0\}$ if $n\geq 5.$ It follows from~\cite{Mar2} and~\cite{CS} that for $n=3,4$ and any $\vare>0$ we have $M_n\cap(\vare,\infty)$ is a finite set. 
 
We will give an effective statement here. Let us mention before proceeding to the relevant statement that we have been informed that; using techniques similar to the work of J.~Cogdell, I.~Piateski-Shapiro and P.~Sarnak, reported on in~\cite{Co}, one can give an alternative quantitative treatment of discreteness of Markov Spectrum. It is quite likely  that the bound one gets using this method could be better than the bound we obtain here.

Let $V$ and $\vare>0$ be given positive numbers, with the understanding that $\vare$ is small and $V$ is large. Let \begin{equation}\label{Hinvariant}\acal(V,\vare)=\{Hx:\h Hx\h\mbox{is closed},\h V\leq\vol(Hx)\leq 2V\h\mbox{and}\h|B(g_x\bbz^3)|^3>\vare\}\end{equation}
Note that $\acal(V,\vare)$ is a finite set. We let $N(V,\vare)=|\acal(V,\vare)|$ be the cardinality of this set, and let $M=N(V,\vare)V.$ Indeed $M\leq\vol(\acal(V,\vare))\leq 2M.$

\begin{cor}\label{cor;markov}
With the notation as above we have 
\[\vol(\acal(V,\vare))\leq C\vare^{-\eta}\]
where $C$ is an absolute constant. Furthermore we can let $\eta=784\times 10^5.$ 
\end{cor}
\begin{proof}
Let $y=\bbz^3\in X.$ Let $f_y$ be a smooth bump function supported in $\vare^{1/3}$ neighborhood of $y$ such that $0\leq f_y\leq 1$ and $f_y=1$ in the ball of radius $\vare^{1/3}/2.$ We may and will choose such $f_y$ so that $\scal_{20}(f_y)\leq\vare^{-20/3}$ and $\vare^{8/3}/\iota\leq m(f)\leq \iota\vare^{8/3}$ for some absolute constant $\iota>0.$ Now apply Theorem~\ref{thm;eff-equi} to the function $f_y.$ We have: if $M>M_0,$ then there exists some $Hx_0\in\acal(V,\vare)$ such that  
\begin{equation}\label{e;markov2}|\mu_0(f)-m(f)|<c\hh\scal_{20}(f_y)M^{-\delta}\end{equation}  
where $\mu_0$ is the $H$-invariant probability measure on $Hx_0$ and $c$ is a universal constant. The inequality~\eqref{e;markov2} together with the bounds on $\scal_{20}(f_y)$ and $m(f_y)$ implies that if $M>\vare^{-\eta},$ with the given $\eta,$ then $\mu_0(f_y)>0.$ In particular we have $Hx_0$ intersects the ball of radius $\vare^{1/3}$ about $y,$ i.e. $Hx_0\not\in\acal(V,\vare)$ which is a contradiction.
\end{proof}

It is worth mentioning that Corollary~\ref{cor;markov} and a geometric series argument imply that $\# (M_3\cap(\vare,\infty))\leq C\vare^{-\eta},$ where $C$ is a (computable) absolute constant and $\eta=784\times 10^5.$



\end{document}